\documentclass[11pt, oneside]{article}   	
\usepackage{geometry}                		
\usepackage{graphicx}				
								
\usepackage[all]{xy}
\CompileMatrices

\usepackage{amssymb}
\usepackage{amsmath}
\usepackage{stmaryrd}
\usepackage{amsthm}
\usepackage{tikz-cd}
\usepackage{extarrows}
\usepackage[mathscr]{euscript}
\usepackage{mathrsfs} 
\usepackage{verbatim}
\usepackage[OT2,T1]{fontenc}

\usepackage[noinputenc,nocaptions]{vietnam}
\catcode`\@=11
\def\tabb{\@tabacckludge}
\catcode`\@=12

\usepackage{rotating}

\usepackage{bm}

\DeclareSymbolFont{cyrletters}{OT2}{wncyr}{m}{n}
\DeclareMathSymbol{\Sha}{\mathalpha}{cyrletters}{"58}
\DeclareMathSymbol{\Che}{\mathalpha}{cyrletters}{"51}

\usepackage{calligra}
\usepackage{mathrsfs}

\newcommand{\Gm}{{\mathbf{G}}_{\rm{m}}}

\DeclareMathOperator{\Gal}{Gal}

\DeclareMathOperator{\Br}{Br}

\DeclareMathOperator{\Pic}{Pic}

\DeclareMathOperator{\R}{R}

\DeclareMathOperator{\Spec}{Spec}

\DeclareMathOperator{\ind}{ind}
\DeclareMathOperator{\per}{per}

\newcommand*{\Z}{\ensuremath{\mathbf{Z}}}                        
\newcommand*{\Q}{\ensuremath{\mathbf{Q}}}                     
\newcommand*{\C}{\ensuremath{\mathbf{C}}}                        
\newcommand*{\A}{\ensuremath{\mathbf{A}}}                        
\newcommand*{\calO}{\mathcal{O}}                                  
\newcommand*{\address}{Einstein Institute of Mathematics, The Hebrew University of Jerusalem, Edmond J. Safra Campus, 91904, Jerusalem, Israel}
\newcommand*{\email}{zevrosengarten@gmail.com}

\newtheorem{theorem}{Theorem}[section]

\newtheorem{lemma}[theorem]{Lemma}
\newtheorem{proposition}[theorem]{Proposition}
\newtheorem{corollary}[theorem]{Corollary}

\theoremstyle{definition}
  \newtheorem{definition}[theorem]{Definition}

\theoremstyle{remark}
  \newtheorem{remark}[theorem]{Remark}

\theoremstyle{definition}

\theoremstyle{remark}

\usepackage[OT2,T1]{fontenc}

\tikzset{commutative diagrams/.cd,
mysymbol/.style={start anchor=center,end anchor=center,draw=none}
}

\usepackage{stmaryrd}

\title{\textbf{SPLITTING $p$-PRIMARY COHOMOLOGY CLASSES OF TORI IN CHARACTERISTIC $p$}} 
\author{Zev Rosengarten \thanks{
While completing this work, the author was supported by Israel Science Foundation Grant No.\,2083/24.
}}

\begin{document}

\date{}
\maketitle

\begin{abstract}
We prove that $p$-primary cohomology classes of a torus $T$ over a global function field of characteristic $p$ may be split by suitable separable $p$-primary extensions. More precisely, we show that such cohomology classes will split in any ``large'' $p$-primary extension (and in fact, prove the same for $\ell$-primary classes over ``large'' $\ell$-primary extensions for every prime $\ell$, including $\ell \neq \mathrm{char}(K)$), and we prove that $p^n$-torsion classes may be split by a (solvable) separable $p$-primary extension of degree $\leq (p^n)^{1+cm\mathrm{log}(m)^3}$ for an explicitly computable universal constant $c > 0$, where $m$ is the degree of a finite Galois extension splitting the torus $T$. Along the way, we also prove Grunwald-Wang type results of independent interest which allow one to approximate a given finite list of abelian $p$-primary local extensions of places of a global function field by a suitable global extension.
\end{abstract}

\tableofcontents{}

\section{Introduction}

Given an algebraic variety $X$ over a field $K$, one of the most basic questions is whether $X$ admits a rational point over $K$. A related question is to determine whether $X$ admits a rational point over certain finite extensions $L/K$. An important quantity which to a certain extent measures the answer to this question is the so-called index $\ind(X/K)$ of $X$, defined to be the greatest common divisor of $[L: K]$ as $L$ varies over those finite extensions over which $X$ acquires points. When $K$ is imperfect, it is natural to also consider the separable index $\ind_{\mathrm{sep}}(X/K)$, defined to be the gcd of those finite {\em separable} $L/K$ such that $X(L) \neq \emptyset$, assuming such extensions exist. When $X$ is generically smooth and nonempty, then $X(K_s) \neq \emptyset$ and so $\ind_{\mathrm{sep}}(X/K) < \infty$. In fact, by a theorem of Gabber-Liu-Lorenzini, when $X$ is a regular and generically smooth $K$-scheme of finite type, one has $\ind(X/K) = \ind_{\mathrm{sep}}(X/K)$ \cite[Th.\,9.2]{gll}.

A case of particular interest is when $G$ is a $K$-group scheme and $X$ is an fppf $G$-torsor over $K$. Then $X$ acquires a point over $L$ precisely when $L$ splits the corresponding cohomology class in $\mathrm{H}^1(K, G)$. (When $G$ is smooth, this agrees with the Galois cohomology.) If $G$ is commutative, then the period $\per(X)$ (the order of $X$ as an element of the group $\mathrm{H}^1(K, G)$) divides the index $\ind(X)$, and it is natural to wonder how far this is from being an equality. When $G = T$ is a torus, a simple argument allows one to construct an extension of degree dividing $\per(X)^d$ ($d := \mathrm{dim}(T)$) splitting $X$ \cite[Prop.\,13.2]{borovoireichstein}, hence $\ind(X)\mid \per(X)^d$. However, when $p := \mathrm{char}(K) > 0$ and $p\mid \ind(X)$, the extension constructed in the proof is typically not separable, so one does not immediately obtain the existence of a separable $p$-primary extension splitting $X$. Nevertheless, thanks to the theorem of Gabber-Liu-Lorenzini, one still obtains that $\ind(X)\mid \per_{\mathrm{sep}}(X)^d$ so that, when $\per(X)$ is a power of $p$, so too is the gcd of the degrees of the separable extensions splitting $X$.

In the present paper, we are concerned with two subtle modifications of the period-index problem for $p$-primary cohomology classes of tori over global function fields of characteristic $p$:
\\

\noindent (1) Can one find a {\em separable} extension of $p$-power degree bounded solely in terms of data attached to the torus (in particular, of course we want it to be independent of the cohomology class) splitting the given class?
\\

\noindent (2) Can one find subextensions of a particular given (infinite) $p$-primary extension of the function field $K$ splitting the given class?
\\

While (2) has the additional feature of finding a splitting field in a given $p$-primary extension, both (1) and (2) differ from the index problem because we are seeking a separable $p$-primary extension splitting the class, whereas the separable index being $p$-primary merely tells us that no prime other than $p$ can divide the degrees of {\em all} of the separable extensions splitting the class.

Before stating our main results, we remark that, by simple arguments, one may show that, over a field $K$ of characteristic $p$, the issue of splitting, over a separable $p$-primary extension, a $p$-primary class in $\mathrm{H}^1(K, G)$ for an arbitrary affine commutative $K$-group scheme $G$ of finite type reduces easily to the torus case, so it is reasonable to concentrate on this setting. (The key point here is that, for a smooth connected commutative unipotent group $U$ in characteristic $p$, $U/E$ is split unipotent for some finite \'etale $E \subset U$ \cite[Lem.\,2.1.5]{Rosengarten}.)

While our main interest in the present paper lies in $p$-torsion phenomena, we will actually prove some more general results, dealing with $\ell$-torsion phenomena for all primes $\ell$, in particular in the case of question (2) above. Before stating our results in this direction, we require some definitions.

\begin{definition}
\label{localptower}
Let $K$ be a field, $\ell$ a prime. We say that an algebraic extension $L/K$ is an {\em $\ell$-primary extension} when,
for every finite subextension $L/F/K$, the degree $[F: K]$ is a power of $\ell$.
\end{definition}

\begin{definition}
Let $K$ be a global function field. We say that an algebraic extension field $L/K$ is {\em locally infinite} when,
for every place $w$ of $L$ lying over a place $v$ of $K$, the extension $L_w/K_v$ is infinite.
\end{definition}

It is easy to see that every global function field admits a (many, in fact) Galois, locally infinite $\ell$-primary extension for every prime $\ell$.
One example is the maximal abelian $\ell$-primary extension. Our first main result is the following theorem, which shows that the separable $\ell$-power extensions of global function fields
killing $\ell$-power torsion cohomology classes of tori are ubiquitous.

\begin{theorem}
\label{charpinfinite}
Let $L$ be a Galois, locally infinite $\ell$-primary extension of a global function field. Then for every $L$-torus $T$, ${\rm{H}}^i(L, T)[\ell^{\infty}] = 0$ for $i > 0$.
\end{theorem}

Here ${\rm{H}}^i(L, T)[\ell^{\infty}]$ denotes the union over $n\ge 1$ of the subgroups
\[{\rm{H}}^i(L, T)[\ell^n]:= \ker\left[ {\rm{H}}^i(L, T)\xrightarrow{[\ell^n]} {\rm{H}}^i(L, T)\right].\]

\begin{remark}
The same result holds in the number field setting, provided that the definition of locally infinite is appropriately modified to account for the archimedean places. Namely, we say that an algebraic $\ell$-primary extension $L$ of a number field $K$ is {\em locally infinite} when the following conditions hold:
\begin{itemize}
\item[(1)] For every place $w$ of $L$ lying over a finite place $v$ of $K$,
the extension $L_w/K_v$ is infinite.
\item[(2)] If $\ell = 2$, then for every place $w$ of $L$ lying over a real place $v$ of $K$,
we have $L_w \simeq \C$.
\end{itemize}
The main reasons that we restrict our attention to the function field setting are (1) the main purpose of this paper is to deal with characteristic $p$ phenomena, and (2) we wish to avoid certain minor technical complications arising from real places.
\end{remark}

\begin{corollary}
\label{charp}
For any global function field $K$, prime $\ell$, $K$-torus $T$ and $\xi \in {\rm{H}}^i(K, T)[\ell^\infty]$ with $i > 0$, there is a finite $($separable$)$ abelian extension $F/K$ of $\ell$-power degree splitting $\xi$.
\end{corollary}

\begin{proof}
The maximal abelian $\ell$-primary extension $L$ of $K$ is a Galois, locally infinite $\ell$-primary extension. By Theorem \ref{charpinfinite}, $\xi$ is split by $L$. Because cohomology of tori (and in fact much more generally) commutes with filtered direct limits of rings \cite[Exp.\,VI, Cor.\,5.2]{sgaiv2}, and because $L$ may be written as the filtered direct limit of its subfields finite over $K$, we deduce that $\xi$ splits over some finite subextension $F/K$ of $L$.
\end{proof}

While Theorem \ref{charpinfinite} is interesting, and gives a strong form of an answer to question (2) above, it does not answer (1). As remarked earlier, a simple argument shows that, over any field, if $T$ is a $d$-dimensional torus, and $\xi \in {\rm{H}}^1(K, T)$ with $p := {\rm{char}}(K) \nmid \per(\xi)$, then there is a finite separable extension $L/K$ of degree dividing $\per(\xi)^d$ that splits $\xi$. We would like to obtain a similar bound in the case when $\per(\xi)$ is a power of $p$. That is, we wish to show that we may split $p^n$-torsion cohomology classes in separable $p$-primary extensions of degree bounded solely in terms of $n$ and the torus $T$.

Suppose that $T$ splits over a finite Galois extension of degree $m$, and let $\xi \in \mathrm{H}^1(K, T)[p^n]$. We will prove that, when $K$ is a global function field, one can find a finite separable extension $L/K$ of degree dividing $(p^n)^{m'}$ which splits $\xi$, where $m'$ depends only on $m$ in a uniform manner (independent of the global field or the torus $T$). If $d := {\rm{dim}}(T)$ as above, then the Galois lattice of $T$ is a representation $\rho\colon {\rm{Gal}}(K_s/K) \rightarrow {\rm{GL}}_d(\Z)$. The torus $T$ splits over the fixed field of $\ker(\rho)$, which has degree equal to the order of a finite subgroup of ${\rm{GL}}_d(\Z)$. By a theorem of Minkowski \cite{minkowski}, this order is bounded (in terms of $d$), hence we see that $m$ is bounded in terms of $d$. Thus our result below yields in particular a bound on the degree of the separable extension killing a cohomology class in terms of the dimension, just as in the case when $p\nmid \per(\xi)$.

\begin{theorem}
\label{charpquantitative}
There is an effectively computable universal constant $c > 0$ with the following property. For any global function field $K$, any $K$-torus $T$ which splits over a Galois extension of degree $m$, and any $\xi \in {\rm{H}}^i(K, T)[p^n]$ with $p := \mathrm{char}(K)$ and  $i > 0$, $\xi$ splits over a finite separable solvable $p$-primary extension of degree $\leq (p^n)^{1+cm{\rm{log}}(m)^3}$.
\end{theorem}

Note that, in a sense, $m$ rather than $d = \mathrm{dim}(T)$ is the natural parameter above, as suggested, for instance, by Hilbert's theorem 90. One could modify the argument to kill $\xi$ in an abelian $p$-primary extension, but one would then obtain from our argument a bound of the shape $p^{f(m, n)}$ for some function $f$ that is linearithmic in $n$ rather than linear as in the case when $p\nmid \per(\xi)$. It is an interesting question to determine what the correct order of growth for the exponent in Theorem \ref{charpquantitative} should be, including for various different types of extensions. (The vanishing of $p$-primary cohomology classes in large separable $p$-primary extensions is a rather general and robust phenomenon, as illustrated by Theorem \ref{charpinfinite}.)

\begin{remark}
\label{vanishesbeyonddeg2}
In the above theorems, the cohomology of $T$ vanishes beyond degree $2$, as follows from the global field case \cite[Proposition 3.2]{Rosengarten},
together with the fact that every pair $(T, \xi)$, where $\xi\in {\rm{H}}^i(L,T)$ with $L/K$ algebraic,
descends to some subextension of $L$ which is a global field, which follows from \cite[Exp.\,VI, Cor.\,5.2]{sgaiv2}.  Thus the only interesting cases are when $i = 1, 2$.
\end{remark}

\begin{remark}
Although we will prove Theorems \ref{charpinfinite} and \ref{charpquantitative} together, most of the work comes from the latter. Thus, for the reader who wishes to skip this and concentrate only on \ref{charpinfinite}, much of the presentation in the paper could be skipped. In particular, one could skip all of the intermediate results before the proof of Theorem \ref{charpinfinitemult} is given except for Lemma \ref{exhaust Sha}(i) sans the bound on $[F: K]$, Lemma \ref{Brlemma}, and case (1) in Propositions \ref{canlandinsha} and \ref{killingsha}.
\end{remark}

We will prove the above results by establishing analogous statements for ${\rm{H}}^2(K, M)$ with $M$ a finite multiplicative group scheme -- that is, a finite group scheme of multiplicative type, from which the statements for tori will follow easily.

\begin{theorem}
\label{charpinfinitemult}
Let $L$ be a Galois, locally infinite $\ell$-primary extension of a global function field. Then for every finite multiplicative $L$-group scheme $M$ of $\ell$-power order, ${\rm{H}}^2(L, M) = 0$.
\end{theorem}

\begin{theorem} \label{charpinfinitequantitativemult}
There is an effectively computable universal constant $c > 0$ with the following property: For any finite multiplicative $K$-group scheme $M$ over a global function field $K$ such that $M$ splits over a finite Galois extension of degree $\leq m$, and any $\xi \in {\rm{H}}^2(K, M)[p^n]$ with $p := \mathrm{char}(K)$, $\xi$ splits over a finite separable solvable $p$-primary extension of degree $\leq (p^n)^{1+cm{\rm{log}}(m)^3}$.
\end{theorem}

The basic idea behind the proofs of all of our main theorems above is as follows. Let $M$ be a finite multiplicative group scheme over a global function field $K$, and suppose given a cohomology class $\xi \in \mathrm{H}^2(K, M)[p^n]$. The proof that $\xi$ vanishes in a suitably large extension $L/K$ then proceeds in two steps. First, we prove that by passing to a finite extension $F/K$ we can ensure that $\xi_F \in \Sha^2(F, M)$, the Tate-Shafarevich group. Then we show that we can split $\xi$ by passing to a further finite extension. At the heart of both steps lies \v{C}esnavi\v{c}ius' extension of Tate duality to general finite commutative group schemes over local and global function fields \cite{cesnavicius}.

To obtain $\xi_F \in \Sha^2(F, M)$, we first note that by general results, $\xi$ is automatically trivial in $\mathrm{H}^2(K_v, M)$ for all but finitely many places $v$ of $K$. For the finite list of places at which $\xi$ does not vanish, we prove the existence of a finite subextension $E^v/K_v$ of the extension $L_w/K_v$ (where $w\mid v$ is a place of $L$) with the property that $\widehat{M}(E^v) = \widehat{M}(L_w)$, where $\widehat{M}$ is the (\'etale) Cartier dual of $M$. We can then pass to a further finite subextension $L_w/F^v/E^v$ to ensure that $\xi$ pairs trivially with every element of $\widehat{M}(E^v) = \widehat{M}(F^v)$ under the local duality (cup product) pairing. Tate local duality then implies that $\xi_{F^v} = 0$. Then we prove Grunwald-Wang type results which prove the existence of a suitable global subextension $L/F/K$ realizing all of these (finitely many) local extensions to conclude that $\xi_F$ vanishes everywhere locally.

In step two, to obtain the splitting under the further assumption that $\xi \in \Sha^2(K, M)$, we prove the existence of a finite subextension $L/E/K$ such that, for every finite subextension $L/F/E$, the image of the map $\mathrm{H}^1(E, \widehat{M}) \rightarrow \mathrm{H}^1(F, \widehat{M})$ contains $\Sha^1(F, \widehat{M})$. An extremely subtle point is that, while there are finitely many cohomology classes $\alpha_i$ over $E$ which hit all of the $\Sha$ classes over larger subextensions, the $\alpha_i$ may fail to lie in $\Sha$ over $E$ (even though they do over some larger extension), and indeed, may be nontrivial at an infinite set of places of $E$. Especially in the setting of our main quantitative theorems, this presents a problem because, while we can bound the degree of each local extension needed to kill $\alpha_i$ at a given place, we have infinitely many local extensions, so there is no guarantee that we can approximate them all simultaneously by a global extension efficiently (that is, by a global extension of suitably bounded degree).

This issue presents some serious technical difficulties, but to convey the basic idea, let us assume that all of the classes $\alpha_i$ lie in $\Sha^1(E, \widehat{M})$. Then we may use the perfect global duality pairing $$\Sha^2(E, M) \times \Sha^1(E, \widehat{M}) \rightarrow \Q/\Z$$ to pair $\xi$ with all of the $\alpha_i$. It is then easy to show that, upon passage to a suitable larger subextension $L/F/K$, $\xi$ pairs trivially with all of the $\alpha_i$, and because the $\alpha_i$ yield all of the possible $\Sha$ classes even over the larger subextension $F$ of $L$, it follows that $\xi_F$ pairs trivially with all of $\Sha^1(F, \widehat{M})$, and therefore vanishes.

We now briefly summarize the contents of the paper. In \S\ref{grunwaldwangsection}, we prove strengthened versions of the classical Grunwald-Wang theorem on realizing a prescribed finite list of abelian local extensions via a global extension when the extensions in question are $p$-primary. In \S\ref{exhaustionsection}, we prove a couple of crucial lemmas that allow us to pull down phenomena in potentially infinite algebraic extensions to finite subextensions (for instance, the step above in which we say that there is a finite local subextension $L_w/E^v/K_v$ realizing all of the $L_w$-points of $\widehat{M}$). Finally, in \S\ref{splittingclassessection}, we prove our main results using the two-step approach outlined above.

\subsection*{Acknowledgements}
I thank Mikhail Borovoi and Zinovy Reichstein for discussions which led to the question that is the topic of the present paper.

\section{$p$-primary Grunwald-Wang phenomena}
\label{grunwaldwangsection}

Essential to the proofs of our main quantitative theorems will be an analogue of the classical Grunwald-Wang theorem, which asserts that, for a global field $K$ and a 
finite set $S$ of places of $K$, given finite cyclic extensions $E^v/K_v$ for each $v \in S$, there is a cyclic extension of global fields $E/K$ of degree equal to the least 
common multiple of the $[E^v: K_v]$, or twice this number, realizing all of these extensions -- that is, there is for each $v \in S$ a place $w\mid v$ of $E$ and a $K_v$-
isomorphism $E_w \simeq E^v$ (equivalently, because $E/K$ is Galois, this holds for every place $w\mid v$ of $E$); see \cite[Ch.\,X, \S2, Th.\,5]{artintate}, which gives 
an adelic formulation of this result.

We shall require a stronger version of this result in characteristic $p$ (so $K$ is a global function field) when the extensions are all $p$-primary and are merely assumed to be abelian (rather than cyclic). In fact, we shall prove (Theorem \ref{realizinglocexts}):
\\

\noindent (1) Given a finite abelian $p$-primary extension $F/K$, finitely many $v \in S$, and $p$-primary extensions $E^v/F_v$ with $E^v/K_v$ abelian, there is a $p$-primary extension $E/F$ which is abelian over $K$, such that $E$ realizes all of the local extensions, and with the degree of $E$ bounded in terms of the degrees of the local extensions and $[F: K]$.
\\

\noindent (2) In the case $F = K$, if $H$ is a group admitting embeddings from all of the $\Gal(E^v/K_v)$ (a clear necessary condition for the conclusion, as the global Galois group must contain all of the decomposition groups), then one may choose $E$ so that $\Gal(E/K)$ embeds as a subgroup of $H$.
\\

Of course, by class field theory, both assertions boil down to statements about the idele class group, and -- as in \cite{artintate}, it is via the properties of this group that we shall proceed. Throughout the remainder of the present section, $K$ denotes a global function field of characteristic $p$. We denote by $\A_K$, or simply $\A$ when $K$ is clear from context, the ring of adeles of $K$, and for a finite set $S$ of places of $K$, we denote by $\A^S$  the restricted product over $v \notin S$ of the fields $K_v$ with respect to $\calO_v$.
The ideles of $K$ are the units $\A^{\times}$ of $\A$ endowed with their usual topology (restricted product of the $K_v^{\times}$ with respect to the subgroups $\calO_v^{\times}$), and similarly for $(\A^S)^{\times}$. As usual, we regard $K^{\times}$ as sitting inside $(\A^S)^{\times}$ via the diagonal embedding, and for $v\notin S$ we regard $K_v^{\times}$ as a subgroup of $(\A^S)^{\times}$ via the embedding which is the identity on the $v$ factor and $1$ on all other factors. The key will be to study the $p$th power map on $(\A^S)^{\times}/K^{\times}$ -- as well as on subgroups of finite index -- for a finite set $S$ of places of $K$.
This is the object of the next series of lemmas.

\begin{lemma}
\label{pthpowerfinindexsubgp}
Let $B \subset K^{\times}$ be a finitely generated subgroup, and let $v$ be a place of $K$. There is a neighborhood $U_B \subset K_v^{\times}$ of $1$ such that $B \cap (U_B\cdot(K_v^{\times})^p) \subset (K^{\times})^p$.
\end{lemma}

\begin{proof}
The group $B/(B \cap (K^{\times})^p)$ is finitely generated and $p$-torsion, hence finite. Let $R \subset B$
 denote a finite set of representatives for the nonzero elements. Because $K_v/K$ is separable and $R \cap (K^{\times})^p = \emptyset$, one has $R \cap (K_v^{\times})^p = \emptyset$. Since the set of $p$th powers in $K_v^{\times}$ is closed, there is a neighborhood $U \subset K_v^{\times}$ of $1$ such that $R \cap (U(K_v^{\times})^p) = \emptyset$. Thus, if $b \in B\backslash(K^{\times})^p$, then $b \in R(K_v^{\times})^p$, so $b \notin U(K_v^{\times})^p$.
\end{proof}

\begin{lemma}
\label{cokerptorsfree}
The group $(\A^S)^{\times}/\overline{K^{\times}}$ is $p$-torsion-free, where $\overline{K^{\times}}$ denotes the closure.
\end{lemma}

\begin{proof}
Suppose that $x \in (\A^S)^{\times}$, and suppose that one has a sequence $\{\alpha_n\}_{n > 0}$ with $\alpha_n \in K^{\times}$ such that $\alpha_n \rightarrow x^p$. We must show that $x \in \overline{K^{\times}}$. First, because they form a Cauchy sequence, there is a finite set $S'$ of places of $K$ such that the $\alpha_n$ are all $S'$-units. It follows that they lie in a finitely generated subgroup $B \subset K^{\times}$. Choosing any place $v \notin S$ of $K$, Lemma \ref{pthpowerfinindexsubgp} then implies that $\alpha_n \in (K^{\times})^p$ for all $n$ sufficiently large, say $\alpha_n = \beta_n^p$. Then the sequence $\{\beta_n\}$ converges to a $p$th root of $x^p$, of which there is only one, so $x \in \overline{K^{\times}}$, as required.
\end{proof}

\begin{lemma}
\label{pthpowervalfg}
For any finite set $S$ of places of $K$, consider the map $f_S\colon K^{\times}/(K^{\times})^p \rightarrow \oplus_{v \notin S} \Z/p\Z$ which on the $v$ component is the valuation $v$ modulo $p$. Then the group $ker(f_S)$ is finite.
\end{lemma}

\begin{proof}
We are free to enlarge $S$. In particular, letting $X$ denote the nonsingular projective curve of which $K$ is the function field, we may assume that $\Pic(X-S) = 0$. We claim that $\ker(f_S)$ equals the image in $K^{\times}/(K^{\times})^p$ of the group $\calO_S^{\times}$ of $S$-units, which is finitely generated. Because $\ker(f_S)$ is torsion, it will then follow that it is finite.

Let $\beta \in K^{\times}$ represent an element of $\ker(f_S)$. Because $\Pic(X - S) = 0$, for each $v \notin S$ there exists $\alpha_v \in K^{\times}$ such that $v(\alpha_v) = 1$ but $w(\alpha_v) = 0$ for all $w \notin S\cup\{v\}$. We then have $\beta(\prod_{v \notin S} \alpha_v^{n_v})^p \in \calO_S^{\times}$ for some integers $n_v$ all but finitely many of which vanish, which proves the claim.
\end{proof}

\begin{lemma}
\label{ppowerclosed}
The $p$th power map $[p]\colon (\A^S)^{\times}/\overline{K^{\times}} \rightarrow (\A^S)^{\times}/\overline{K^{\times}}$ is a homeomorphism onto a closed subgroup.
\end{lemma}

\begin{proof}
By Lemma \ref{cokerptorsfree}, the map is an isomorphism onto its image. Due to \cite[Ch.\,IX,\S5, Prop.\,6, Cor.\,]{bourbakitopology},
 in order to show that it is a homeomorphism onto a closed subgroup, it suffices to prove that $[p]$ has closed image. That is, we must show that $((\A^S)^{\times})^p\overline{K^{\times}} \subset (\A^S)^{\times}$ is closed. Let $x \in (\A^S)^{\times}$ be in the closure, so there is a sequence $\{y_n^p\alpha_n\}_{n > 0}$ converging to $x$, with $y_n \in (\A^S)^{\times}$ and $\alpha_n \in K^{\times}$. Because the sequence is Cauchy, one deduces that there is some $n$ such that, for all $m > 0$, the  ratios $\alpha_{n+m}/\alpha_n$ lie in $\ker(f_S)$ in the notation of Lemma \ref{pthpowervalfg}. That lemma then implies that these ratios lie in a finite set modulo $(K^{\times})^p$. Passing to a subsequence, therefore, we may assume that there exist $\alpha, \beta_n \in K^{\times}$ such that $\alpha_n = \alpha(\beta_n)^p$ for all $n$. We are free to modify $x$ by an element of $K^{\times}$. Replacing $x$ by $\alpha^{-1}x$, we then obtain that $x \in \overline{((\A^S)^{\times})^p}$. But because $(K_v^{\times})^p \subset K_v^{\times}$ is closed for all $v$, the group $((\A^S)^{\times})^p \subset (\A^S)^{\times}$ is closed, so the proof is complete.
\end{proof}

\begin{lemma}
\label{groupcontainingothers}
For any integer $n \geq 0$ and any prime number $p$, there is an abelian group of order $p^{\lfloor n(1+{\rm{log}}(n))\rfloor}$ containing (an isomorphic copy of) every abelian $p$-group of order $\leq p^n$.
\end{lemma}

In the above lemma, we interpret $0{\rm{log}}(0)$ to be $0$. Similar conventions apply below.

\begin{proof}
We claim that the group $A := \oplus_{i=1}^n (\Z/p^{\lfloor n/i \rfloor}\Z)$ has the desired property. To see this, let $B = \sum_{j=1}^n (\Z/p^j\Z)^{e_j}$ be an abelian group of order $\leq p^n$, so
\begin{equation}
\label{groupcontainingotherspfeqn1}
\sum_{j=1}^n je_j \leq n.
\end{equation}
For an abelian $p$-group $C$, let $f_m(C)$ denote the number of summands of $C$ isomorphic to $\Z/p^m\Z$. Then the condition for $C$ to be isomorphic to a subgroup of $C'$ is that, for every integer $m > 0$,
\[
\sum_{j \geq m} f_j(C) \leq \sum_{j \geq m} f_j(C').
\]
Thus, in order to show that $A$ has the desired property, we must show that -- under the assumption (\ref{groupcontainingotherspfeqn1}) -- one has for all $m > 0$,
\begin{equation}
\label{groupcontainingotherspfeqn2}
\sum_{j \geq m} e_j \leq \sum_{i\colon\lfloor n/i\rfloor \geq m} 1,
\end{equation}
or equivalently,
\[
\sum_{j \geq m} e_j \leq n/m.
\]
But we have -- using (\ref{groupcontainingotherspfeqn1}) --
\[
n \geq \sum_{j \geq m} je_j \geq m\sum_{j \geq m} e_j,
\]
which yields (\ref{groupcontainingotherspfeqn2}).

It now only remains to show that $A$ has order bounded by $p^{\lfloor n(1+{\rm{log}}(n))\rfloor}$. We have
\[
{\rm{log}}_p(\#A) = \sum_{i=1}^n\lfloor n/i\rfloor \leq n\sum_{i=1}^n 1/i \leq n(1+{\rm{log}}(n)).
\]
\end{proof}

\begin{lemma}
\label{subgpofboundedindex}
Let $p$ be a prime number, $n > 0$ an integer, and for $1 \leq i \leq m$, let $A_i$ be a finite abelian $p$-group. Let $D \subset A := \prod_{i=1}^r A_i$ be a subgroup of index $\leq p^m$. For each $i$, let $D_i := D \cap A_i$, where we regard $A_i$ as a subgroup of $A$ via the map which is the identity on the $i$ factor and $0$ on the other factors.
\begin{itemize}
\item[(i)] Suppose that $\#D_i \leq p^n$ for all $i$. Then there is a subgroup $E \subset D$ with $[D: E] \leq p^{\lfloor (n+m)(1+{\rm{log}}(n+m))\rfloor}$ such that $E \cap D_i = 0$ for all $1 \leq i \leq r$.
\item[(ii)] If $m = 0$ $($so $D = A$$)$, and $H$ is a finite abelian group admitting an inclusion from each $A_i$, then there is a subgroup $E \subset A$ such that there is an embedding $E/A \hookrightarrow H$ and such that $E \cap A_i = 0$ for all $1 \leq i \leq r$.
\end{itemize}
\end{lemma}

The key point of the above lemma, from our perspective, is that the index $[D: E]$ may be bounded solely in terms of the orders $\#D_i$ and the index $[A: D]$, and independently of the number $r$ of factors in $A$.

\begin{proof}
In case (i), one has for each $i$ an inclusion $A_i/D_i \hookrightarrow A/D$, hence $$\#A_i = [A_i: D_i]\cdot \#D_i \leq [A: D]\cdot \#D_i \leq p^{n+m}.$$ Let $H$ be an abelian group admitting for each $i$ an inclusion $\phi_i\colon A_i \hookrightarrow H$ (taken to be the same $H$ as in (ii) in case (ii)). By Lemma \ref{groupcontainingothers}, we may in case (i) choose $H$ so that $\#H = p^{\lfloor (n+m)(1+{\rm{log}}(n+m)) \rfloor}$. Let $\phi\colon A \rightarrow H$ be the homomorphism which is $\phi_i$ on $A_i$, and let $E := \ker(\phi|_D)$. Then $E \cap D_i = 0$ for all $i$ because $\phi_i$ is an inclusion, which proves (ii), while $D/E$ is a subgroup of $H$, hence is a $p$-group of order $\leq p^{(n+m)(1+{\rm{log}}(n+m))}$, which proves (i).
\end{proof}

\begin{lemma}
\label{extendingchars}
Suppose given a short exact sequence of locally compact Hausdorff abelian groups
\[
0 \longrightarrow G' \xlongrightarrow{i} G \xlongrightarrow{\pi} G'' \longrightarrow 0
\]
such that $i$ is a topological isomorphism onto a closed subgroup of $G$, and $\pi$ induces an isomorphism $G/i(G') \simeq G''$. Suppose that the multiplication by $p$ map $[p]_{G''}\colon G'' \rightarrow G''$ is a topological isomorphism onto a closed subgroup of $G''$. Then any continuous homomorphism $\chi\colon G' \rightarrow H$ into a finite discrete abelian $p$-group extends to a continuous homomorphism $G \rightarrow H$.
\end{lemma}

\begin{proof}
We may assume that $H = p^{-r}\Z/\Z \subset \R/\Z$. Then the character $\chi \in \widehat{G'}$ extends to a character $\psi \in \widehat{G}$, and $p^r\psi = \phi\circ\pi$ for some $\phi \in \widehat{G''}$. Because $[p^r]_{G''}$ is a topological isomorphism onto a closed subgroup, the character $\phi$ extends along $[p^r]_{G''}$. That is, there exists $\zeta \in \widehat{G''}$ such that $\zeta \circ [p^r]_{G''} = \phi$. Then $p^r(\psi - \zeta\circ\pi) = 0$, so $\psi-\zeta\circ\pi \in \widehat{G}$ is an extension of $\chi$ that lands in $p^{-r}\Z/\Z$.
\end{proof}

\begin{proposition}
\label{effapproxadelicversion}
Let $S$ be a finite set of places of $K$, let $C \subset \A^{\times}/K^{\times}$ be a closed subgroup of index $p^m$, and suppose given for every $v \in S$ a closed subgroup $G_v \subset C \cap K_v^{\times}$ of $p$-power index $\leq p^n$.
\begin{itemize}
\item[(i)] There is a closed subgroup $B \subset C$ with $[C: B] \mid p^{\lfloor(n+m)(1+{\rm{log}}(n+m))\rfloor}$ such that $B \cap K_v^{\times} = G_v$ for every $v \in S$.
\item[(ii)] If $m = 0$ $($so $C = \A^{\times}/K^{\times}$$)$, and $H$ is a finite abelian group admitting an inclusion from $K_v^{\times}/G_v$ for all $v \in S$, then there exist a closed subgroup $B \subset \A^{\times}/K^{\times}$ such that there is an embedding $(\A^{\times}/K^{\times})/B \hookrightarrow H$ and such that $B \cap K_v^{\times} = G_v$ for every $v \in S$.
\end{itemize}
\end{proposition}

\begin{proof}
For each $v \in S$, let $A_v := K_v^{\times}/G_v$. Let $A := \prod_{v \in S} A_v$, and let $D \subset A$ be the image in $A$ of $C \cap (\prod_{v \in S} K_v^{\times})$. Applying Lemma \ref{subgpofboundedindex}, we obtain a subgroup $E \subset D$ of index dividing $p^{\lfloor(n+m)(1+{\rm{log}}(n+m))\rfloor}$ in case (i), and such that there is an inclusion $D/E \hookrightarrow H$ in case (ii), such that $E \cap D_v = 0$ for all $v \in S$. Let $B'$ be the preimage in $C\cap (\prod_{v \in S} K_v^{\times})$ of $E$. Then $B'$ is closed in $C\cap (\prod_{v \in S} K_v^{\times})$,
\begin{equation}
\label{effapproxadelicversionpfeqn1}
B' \cap K_v^{\times} = G_v
\end{equation}
for every $v \in S$, and
\[
\left[C\cap (\prod_{v \in S} K_v^{\times}): B'\right] \bigg| \hspace{.03 in} p^{\lfloor(n+m)(1+{\rm{log}}(n+m))\rfloor}.
\]

Now consider the following commutative diagram of exact sequences of locally compact Hausdorff abelian groups, where $R$ and $Q$ are defined to be the quotient groups making the sequences exact, and where the right vertical arrow is an inclusion:
\[
\begin{tikzcd}
0 \arrow{r} & C \cap \overline{\prod_{v \in S} K_v^{\times}} \arrow{r} \arrow{d} & C \arrow{r}{\pi} \arrow{d} & R \arrow{r} \arrow[d, hookrightarrow] & 0 \\
0 \arrow{r} & \overline{\prod_{v \in S} K_v^{\times}} \arrow{r} & \A^{\times}/K^{\times} \arrow{r} & Q \arrow{r} & 0
\end{tikzcd}
\]
The group $Q$ is the maximal Hausdorff quotient of $\A^{\times}/K^{\times}\prod_{v \in S}K_v^{\times} = (\A^S)^{\times}/K^{\times}$, which is $(\A^S)^{\times}/\overline{K^{\times}}$, where now the closure is in $(\A^S)^{\times}$. It follows from Lemma \ref{ppowerclosed} that the multiplication by $p$ map $[p]_Q\colon Q \rightarrow Q$ induces a topological isomorphism onto a closed subgroup. We claim that the same statement holds with $Q$ replaced by $R$: The multiplication by $p$ map $[p]_R\colon R \rightarrow R$ induces a topological isomorphism onto a closed subgroup. Indeed, from the corresponding statement for $Q$, it suffices to check that the inclusion $R \hookrightarrow Q$ is a homeomorphism onto a closed subgroup. By \cite[Ch.\,IX,\S5, Prop.\,6, Cor.\,]{bourbakitopology}, it suffices merely to verify that the image is closed.

For the required closedness, we must check that $\A^{\times}/K^{\times}C\overline{\prod_{v \in S}K_v^{\times}}$ is Hausdorff. As observed above, $Q = (\A^S)^{\times}/\overline{K^{\times}}$, so the question becomes the following: Given an inclusion $K^{\times} \subset G \subset (\A^S)^{\times}$ with $G$ closed of finite index, show that $G\cdot\overline{K^{\times}} \subset (\A^S)^{\times}$ closed. Let $x \in (\A^S)^{\times}$ lie in the closure of $G\overline{K^{\times}}$. Thus there exists a sequence $\{\alpha_ng_n\}_{n>0}$ tending to $x$ with $\alpha_n \in K^{\times}$ and $g_n \in G$. Because $G \subset (\A^S)^{\times}$ is of finite index, $K^{\times}/(G \cap K^{\times})$ is finite. Passing to a subsequence, therefore, we may assume that the $\alpha_n$ all agree modulo $G \cap K^{\times}$. Then replacing $(\alpha_n, g_n)$ by $(\alpha_1, (\alpha_1^{-1}\alpha_n)g_n)$, we may assume that the $\alpha_n$ all agree -- say,  $\alpha_n = \alpha$. Since we are free to modify $x$ by an element of $K^{\times}$, we may therefore assume that all $\alpha_n = 1$. Then $x \in \overline{G} = G$. This completes the proof that $R$ is closed in $Q$ (and hence that $[p]_R$ is a homeomorphism onto a closed subgroup).

Now we have the closed finite index subgroup $B' \subset C\cap(\prod_{v \in S} K_v^{\times})$. Because $B'$ is closed of finite index, the map
\[
\frac{C\cap(\prod_{v \in S} K_v^{\times})}{B'} \rightarrow \frac{\overline{C\cap(\prod_{v \in S} K_v^{\times})}}{\overline{B'}}
\]
is an isomorphism. But because $C$ is open and closed, we have an equality
\[
\overline{C \cap (\prod_{v \in S}K_v^{\times})} = C \cap \overline{\prod_{v \in S} K_v^{\times}}.
\]
Thus $\overline{B'} \subset C \cap \overline{\prod_{v \in S}K_v^{\times}}$ is a closed subgroup of index dividing $p^{\lfloor(n+m)(1+{\rm{log}}(n+m))\rfloor}$.

By Lemma \ref{extendingchars}, we have a commutative diagram of continuous homomorphisms
\[
\begin{tikzcd}
C \cap \overline{\prod_{v \in S} K_v^{\times}} \arrow{dr} \arrow{r} & C \arrow{d}{\chi} \\
& \left[C \cap \overline{\prod_{v \in S} K_v^{\times}}\right]/\overline{B'}
\end{tikzcd}
\]
We claim that $B := \ker(\chi)$ satisfies the requirements of the proposition. Indeed, $B \subset C$ is closed of index dividing $\#[C \cap \overline{\prod_{v \in S} K_v^{\times}}]/\overline{B'}\mid p^{\lfloor(n+m)(1+{\rm{log}}(n+m))\rfloor}$, and $$B \cap K_v^{\times} = \overline{B'}\cap K_v^{\times} = B' \cap K_v^{\times} = G_v,$$ the penultimate equation holding because $B'$ is closed in $C \cap(\prod_{v \in S}K_v^{\times})$, and the final equation being (\ref{effapproxadelicversionpfeqn1}).
\end{proof}

Now we come to the following result, interesting in its own right, which will play a crucial role in the proof of Theorem \ref{charpquantitative}.

\begin{theorem}
\label{realizinglocexts}
Let $K$ be a global function field of characteristic $p$, and let $F/K$ be a finite abelian extension of degree $p^m$. Let $S$ be a finite set of places of $F$. Suppose given for each $v \in S$ a $p$-primary extension $E^v/F_v$ of degree $\leq p^n$ such that $E^v$ is abelian over $K_w$, where $v\mid w$. Then there is a $p$-primary extension $E/F$ -- abelian over $K$ -- such that, for each place $v'$ of $E$ lying above a place $v \in S$, $E_{v'}$ and $E^v$ are $F_v$-isomorphic, and such that
\begin{itemize}
\item[(i)] $[E: F] \leq p^{(n+m)(1+{\rm{log}}(n+m))}$.
\item[(ii)] If $m = 0$ $($so $F = K$$)$, and one has for each $v \in S$ an inclusion $\Gal(E^v/K_w) \hookrightarrow H$, then there is an inclusion $\Gal(E/K) \hookrightarrow H$.
\end{itemize}
\end{theorem}

\begin{proof}
This is exactly the statement of Proposition \ref{effapproxadelicversion} translated from adelic language via local and global class field theory (and via the compatibility between the two).
\end{proof}

Although we will not require it, the following corollary is sufficiently interesting that we record it.

\begin{corollary}
\label{realizinglocextsgeneralcase}
Let $K$ be a global function field of characteristic $p$, and let $K^{\rm{p-sol}}$ be the maximal solvable $p$-primary extension of $K$. Let $S$ be a finite set of places of $K$, and suppose given for each $v \in S$ a finite $p$-primary extension $F^v/K_v$ of degree $\leq p^n$. Then there is a subextension $K^{\rm{p-sol}}/F/K$ with $[F: K] \leq p^{\frac{1}{2}n(n+1)(1+{\rm{log}}(n))}$ such that, for each place $v'$ of $F$ lying above a place $v \in S$, $F_{v'}$ and $F^v$ are $K_v$-isomorphic.
\end{corollary}

\begin{proof}
The Galois group of every local field is solvable, so we have for each $v \in S$ a filtration $K_v = F^v_0 \subset F^v_1 \subset \dots \subset F^v_n = F^v$ with $F^v_i/F^v_{i-1}$ abelian for each $i$, and such that, if $F^v_i = F^v_{i-1}$, then $F^v_j = F^v_{j-1}$ for all $j \geq i$. In particular, $[F^v_i: F^v_{i-1}] \leq p^{n-(i-1)}$. Iteratively applying Theorem \ref{realizinglocexts}, we obtain a tower of abelian extensions $K = F_0 \subset F_1 \subset \dots \subset F_n =: F$ such that -- at each stage -- $F_i/F_{i-1}$ ``realizes'' the local extension $F^v_i/F^v_{i-1}$ for each $v \in S$, and $[F_i: F_{i-1}] \leq p^{(n-(i-1))(1+{\rm{log}}(n-(i-1)))}$. Thus
\[
{\rm{log}}_p([F: K]) \leq \sum_{i=1}^n i(1+{\rm{log}}(i)) \leq \sum_{i=1}^n i(1+{\rm{log}}(n)) = \frac{1}{2}n(n+1)(1+{\rm{log}}(n)).
\]
\end{proof}

\section{Exhaustion by finite subextensions}
\label{exhaustionsection}

In order to prove Theorems \ref{charpinfinitemult} and \ref{charpinfinitequantitativemult}, we will need to know that certain phenomena occurring in infinite algebraic extensions may in fact be brought down to some finite subextension. In this section we prove a couple of lemmas that are variations on this theme.

\begin{lemma}
\label{exhaustingpts}
Let $L/K$ be an algebraic extension, and let $F/K$ be a finite separable extension of degree $m$. Then there is an intermediate extension $L/M/K$ with $[M: K] \leq m!$ such that $E(L) = E(M)$ for any finite \'etale $K$-group scheme $E$ that splits over $F$.
\end{lemma}

\begin{proof}
For any $E$ as in the lemma, we have the natural inclusion $E \hookrightarrow E' := {\rm{R}}_{F/K}(C)$, where ${\rm{R}}_{F/K}$ denotes Weil restriction and $C := E_F$ is a constant group scheme, and it suffices to prove the conclusion of the lemma for $E'$. We have that $E'(L) = C(F \otimes_K L) = \prod C$, where the product is over the connected components of ${\rm{Spec}}(F \otimes_K L)$. It therefore suffices to find $L/M/K$ with $M/K$ separable, $[M: K] \leq m!$, and such that ${\rm{Spec}}(F \otimes_K L)$ and ${\rm{Spec}}(F \otimes_K M)$ have the same number of components. If $L'$ is the maximal subfield of $L$ that is separable over $K$, then $F \otimes_K L$ and $F \otimes_K L'$ have homeomorphic spectra, hence we may and do assume that $L/K$ is separable.

If we write $F \simeq K[X]/f(X)$, then the number of components of $F \otimes_K L$ is the number of irreducible factors of $f(X)$ over $L$. Let $f(X) = f_1\dots f_k$ be the factorization of $f$ over $L$, and let $d_i := {\rm{deg}}(f_i)$. Let $G_L \subset G_K$ denote the Galois groups of $L$ and $K$, respectively (with respect to some fixed separable closure $L_s$ of $L$). Then $G_L$ acts on the roots $R := \{\alpha_1, \dots, \alpha_m\}$ of $f$ with $k$ orbits of sizes $d_1, \dots, d_k$. Consider the action of $G_K$ on ordered $k$-tuples of disjoint subsets of $R$ of sizes $d_1, \dots, d_k$ which is induced by the permutation action on $R$. Then $G_L$ stabilizes the element corresponding to the roots of $f_1$, the roots of $f_2$, et cetera. Let $G_H \subset G_K$ denote the stabilizer of this element, and $L/H/K$ the corresponding field, so that $[H: K] = [G_K: G_H]$ is at most the number of ordered tuples of disjoint sets as above. This number is readily calculated to be $m!/(d_1!d_2!\dots d_k!) \leq m!$, and over the field $H$, $f$ already has the same factorization $f_1\dots f_k$ that it has over $L$, so that $F\otimes_K H$ has the same number of components as $F \otimes_K L$. This proves the lemma.
\end{proof}

\begin{lemma}
\label{exhaust Sha}
Let $K$ be a global field, let $L/K$ be an algebraic extension, and let $E$ be a finite \'etale commutative $K$-group scheme.
\begin{itemize}
\item[(i)] If $E$ splits over a finite Galois extension $H/K$ of degree $m$, then there exists an intermediate extension $L/F/K$ with $[F: K] \leq m!$ and a finite subgroup $A \subset {\rm{H}}^1(F, E)$ such that, for every intermediate extension $L/F'/F$ with $F'/F$ finite, $\Sha^1(F', E)$ is contained in the image of $A$ under the map ${\rm{H}}^1(F, E) \rightarrow {\rm{H}}^1(F', E)$, and such that every element of $A$ maps to an element of $\Sha^1(M, E)$ for some intermediate extension $L/M/K$ finite over $K$. Furthermore,
\item[(ii)] If $L/K$ is the maximal abelian $p$-primary extension, where $p = {\rm{char}}(K) > 0$, then for any finite set $S$ of places of $F$, there is an $F$-finite subextension $F' = F'_S/F$ of $L$ such that, for some universal and effectively computable $c > 0$,
\begin{equation}
\label{exhaust Shapfeqn7}
{\rm{log}}_p([F': K]) \leq cm{\rm{log}}(m)^2,
\end{equation}
and such that, for every place $w$ of $F'$ above $S$ and every $a \in A$, $a \mapsto 0 \in {\rm{H}}^1(F'_w, E)$.
\end{itemize}
\end{lemma}

\begin{proof}
Let $H/K$ be a finite separable extension of degree $m$ splitting $E$. First, for any finite extension $N/K$, we claim that
\begin{equation}
\label{exhaust Shapfeqn1}
\Sha^1(N, {\rm{R}}_{H/K}(E_H)) = 0.
\end{equation}
Indeed, we have an isomorphism of $K$-algebras $N \otimes_K H \simeq \prod_{i=1}^r N_i$ for some finite field extensions $N_i/K$. Then
\[
\Sha^1(N, {\rm{R}}_{H/K}(E_H)) \simeq \prod_{i=1}^r \Sha^1(N_i, E_H) = 0,
\]
the last equality a consequence of the Cebotarev density theorem because $E_H$ is constant.

Now consider the finite \'etale $K$-group scheme $Q$ defined by the following exact sequence
\begin{equation*}
0 \longrightarrow E \longrightarrow {\rm{R}}_{H/K}(E_H) \xlongrightarrow{\pi} Q \longrightarrow 0,
\end{equation*}
in which the left map is the canonical inclusion. Then $Q$ and ${\rm{R}}_{H/K}(E_H)$ are finite \'etale $K$-group schemes which split over $H$. By Lemma \ref{exhaustingpts}, therefore, one has that $Q(F) = Q(L)$ for some intermediate $L/F/K$ with $[F: K] \leq m!$. It follows from (\ref{exhaust Shapfeqn1}) that, for any intermediate $L/F'/F$ with $F'/F$ finite, every element of $\Sha^1(F', E)$ lies in the image of $B := \delta(Q(F))$ (with $\delta$ the connecting map) under the map ${\rm{H}}^1(F, E) \rightarrow {\rm{H}}^1(F', E)$. To complete the proof of (i), we take $A \subset B$ to be the subgroup consisting of elements which land in $\Sha^1(M, E)$ for some intermediate $F$-finite $L/M/F$.

For (ii), for each $v \in S$ let $u$ be a place of $L$ above $v$. Then we have by Lemma \ref{exhaustingpts} again a subextension $L_u/(F')^v/F_v$ with $[(F')^v: F_v] \leq m!$ -- and hence
\begin{equation}
\label{exhaust Shapfeqn4}
{\rm{log}}_p([(F')^v: F_v]) \leq c_1m{\rm{log}}(m)
\end{equation}
for some universal effectively computable $c_1 > 0$ -- such that $${\rm{R}}_{H/K}(E_H)((F')^v) = {\rm{R}}_{H/K}(E_H)(L_u).$$ Every element of $A$ is $\delta(q)$ for some $q \in Q(F)$, and goes to $0$ in ${\rm{H}}^1(L_u, E)$ (since it lands in $\Sha$ upon passage to some finite subextension of $L$), hence $q$ lifts to an element of ${\rm{R}}_{H/K}(E_H)(L_u) = {\rm{R}}_{H/K}(E_H)((F')^v)$. It follows that $a \mapsto 0 \in {\rm{H}}^1((F')_v, E)$. Because
\begin{equation}
\label{exhaust Shapfeqn6}
{\rm{log}}_p([F: K]) \leq c_1m{\rm{log}}(m),
\end{equation}
Theorem \ref{realizinglocexts} provides a subextension $L/F'/F$ such that $F'_w \simeq (F')^v$ as $F_v$-algebras for every place $w \mid v \in S$ of $F'$, and such that (using (\ref{exhaust Shapfeqn4}))
\begin{equation}
\label{exhaust Shapfeqn5}
{\rm{log}}_p([F': F]) \leq c_2m{\rm{log}}(m)(1+{\rm{log}}(c_2m{\rm{log}}(m))) \leq c_3m{\rm{log}}(m)^2,
\end{equation}
for some universal effectively computable $c_i > 0$. Then every $a \in A$ dies in ${\rm{H}}^1(F'_w, E)$ for every $w \mid S$. Furthermore, combining (\ref{exhaust Shapfeqn5}) and (\ref{exhaust Shapfeqn6}), we obtain (\ref{exhaust Shapfeqn7}).
\end{proof}

\section{Splitting cohomology classes}
\label{splittingclassessection}

In this section we turn to the proofs of Theorems \ref{charpinfinitemult} and \ref{charpinfinitequantitativemult}, and from these we will easily deduce Theorems \ref{charpinfinite} and \ref{charpquantitative}. In the notation of Theorem \ref{charpinfinitemult}, $M$ descends to a group scheme over some intermediate subextension $L/F/K$ finite over $K$, and $L/F$ is then still a Galois, locally infinite $\ell$-primary extension, hence we may assume that $M$ is defined over $K$. Furthermore, any cohomology class $\xi \in {\rm{H}}^2(L, M)$ is also defined over some finite subextension, hence we may also suppose that $\xi \in {\rm{H}}^2(K, M)$, and we wish to show that it dies over $L$. Thus, for the proofs of both theorems, we may assume that $\xi \in {\rm{H}}^2(K, M)[\ell^n]$, where $\ell = p$ in the case of Theorem \ref{charpinfinitequantitativemult}. An easy induction -- and in the case of Theorem \ref{charpinfinitemult} using the fact that $L/E$ is still Galois and locally infinite for any $K$-finite $L/E/K$ -- allows us to assume that $n = 1$. Further, we may write $M = M_{\ell} \times M'$ with $M_{\ell}$ $\ell$-primary and $M'$ of order prime to $\ell$, so we may as well assume that $M = M_{\ell}$ is $\ell$-primary.

We break the proof up into two stages: First, we show that by passing to a suitable finite extension $H/K$, we may ensure that $\xi_H \in \Sha^2(H, M)$. Then we will show that a class in $\Sha^2[\ell]$ may be killed by passing to a further suitable finite extension. We carry out the first step now.

\begin{proposition}
\label{canlandinsha}
Let $K$ be a global function field, and suppose that either (1) $L/K$ is a Galois locally infinite $\ell$-primary extension, or (2) $\ell = p = {\rm{char}}(K)$ and $L/K$ is the maximal abelian $p$-primary extension. Let $M$ be a finite $\ell$-primary multiplicative $K$-group scheme that splits over a finite Galois extension of $K$ of degree $m$, and let $\xi \in {\rm{H}}^2(K, M)[\ell]$. Then there is a subextension $L/H/K$ with $[H: K]$ finite such that $\xi_H \in \Sha^2(H, M)$. In case (2), we may choose $H$ so that
\begin{equation*}
[H: K] \leq p^{1+ cm{\rm{log}}(m)^3},
\end{equation*}
where $c > 0$ is an absolute, effectively computable constant.
\end{proposition}

\begin{proof}
For each place $v$ of $K$, let $\xi_v \in {\rm{H}}^2(K_v, M)$ denote the image of $\xi$. By \cite[Th.\,2.18]{cesnavicius}, one has that $\xi_v = 0$ for all but finitely many $v$. Let $S$ be the finite set of places $v$ of $K$ such that $\xi_v \neq 0$. For each $v \in S$, let $w = w_v$ be a place of $L$ lying above $v$. Let $E := \widehat{M}$ be the Cartier dual, a finite \'etale commutative $K$-group scheme.

By Lemma \ref{exhaustingpts}, there is an intermediate subextension $L_w/F^v/K_v$ with $[F^v: K_v] \leq m!$ such that $E(L_w) = E(F^v)$. We have the Tate local duality pairing
\[
{\rm{H}}^2(F^v, M) \times H^0(F^v, E) \xrightarrow{\cup} {\rm{H}}^2(F^v, \mathbf{G}_m) \xrightarrow{{\rm{inv}}} \Q/\Z,
\]
which is perfect, and similarly for any finite extension of $F^v$ \cite[Prop.\,4.10(b)]{cesnavicius}. Because the invariant of a Brauer class multiplies by the degree when passing to a finite extension, if we pass to a an extension $H^v/F^v$ of $L_w$ of degree divisible by $\ell$, then we find that $(\xi_v)_{H^v}$ annihilates $E(F^v) = E(H^v)$ under the local duality pairing, hence $\xi_v$ dies in $H^v$.

In case (1), choose a subextension $L/F/K$ finite over $K$ such that one has for each $v \in S$ a $K_v$-embedding $F^v \hookrightarrow F \otimes_K K_v$. This means that $F$ has a place $v'\mid v$ such that $F^v$ admits a $K_v$-embedding into $F_{v'}$. Now choose $L/H/F$ with $H/K$ Galois such that, for each $w$, there is a place $w'$ of $H$ above $v'$ with $\ell\mid[H_{w'}: F^v]$, possible because $L/K$ is locally infinite. Then $\xi_v$ dies in $H_{w'}$. Because $H/K$ is Galois, it follows that $\xi_v$ dies over $H_{v''}$ for every place $v''\mid v$ of $H$. Thus $\xi_H \in \Sha^2(H, M)$.

In case (2), one argues similarly, except that we keep track of $[H: K]$. We have
\[
{\rm{log}}_{p}(m!) \leq c'm{\rm{log}}(m)
\]
for some absolute constant $c' > 0$. By Theorem \ref{realizinglocexts}, there is a subextension $L/F/K$ with
\begin{equation}
\label{canlandinshapfeqn1}
{\rm{log}}_p([F: K]) \leq c'm{\rm{log}}(m)(1+{\rm{log}}(c'm{\rm{log}}(m))) \leq c''m{\rm{log}}(m)^2
\end{equation}
such that one has for each place $v'$ of $F$ above a place $v \in S$ a $K_v$-isomorphism $F_{v'} \simeq F^v$. Using Theorem \ref{realizinglocexts} again, one finds that there is a subextension $L/H/F$ such that $p\mid [H_{v''}: F_{v'}]$ for each place $v''$ of $H$ lying above a place $v'$ of $F$ above $S$, and such that
\begin{equation}
\label{canlandinshapfeqn2}
{\rm{log}}_p([H: F]) \leq (1+c''m{\rm{log}}(m)^2)(1+{\rm{log}}(1+c''m{\rm{log}}(m)^2)) \leq 1+ c'''m{\rm{log}}(m)^3.
\end{equation}
Because $L/K$ is abelian, $H/K$ is Galois, hence -- as above -- we conclude that $\xi_H \in \Sha^2(H, M)$. Finally, the desired bound on $[H: K]$ follows from (\ref{canlandinshapfeqn1}) and (\ref{canlandinshapfeqn2}).
\end{proof}

Now we turn to the second stage, which is to find a suitable extension killing an element $\xi \in \Sha^2(K, M)[\ell]$. Before doing so, let us recall the $\Sha$-pairings of global Tate duality.
Let $K$ be a global field, and let $G$ be a finite commutative $K$-group scheme with Cartier dual $\widehat{G}$. For $i = 1, 2$, \cite{cesnaviciussha} defines bi-additive perfect pairings
\begin{equation}
\label{shapairing}
\Sha^i(K, G) \times \Sha^{3-i}(K, \widehat{G}) \rightarrow \Q/\Z.
\end{equation}
A more concrete description of these pairings (in a more general context) is given in \cite[\S5.11]{Rosengarten}.
For the perfection of these pairings, see \cite[Lem.\,4.4(a)]{cesnavicius} or \cite[Th.\,1.2.10]{Rosengarten}. Recall that, for a local field $M$, one has an invariant map ${\rm{H}}^2(M, \mathbf{G}_m) \xrightarrow{{\rm{inv}}} \Q/\Z$. In general, these pairings do not extend to pairings ${\rm{H}}^i(K, G) \times {\rm{H}}^{3-i}(K, \widehat{G})$ of the full cohomology groups, or even to pairings between $\Sha^i$ and ${\rm{H}}^{3-i}$. However, the pairing does extend to this latter pair if we lift to the level of cocycles together with additional data, as we now explain.

(We could also carry out the following procedure with the roles of the indices $1$ and $2$ reversed, but the following is what we will need.) Suppose given a global field $K$, and classes $\eta \in \Sha^2(K, G)$ and $\eta' \in {\rm{H}}^1(K, \widehat{G})$. By \cite[Prop.\,2.9.6]{Rosengarten}, these may be represented by fppf \v{C}ech cocycles $\alpha \in {\check{Z}}^2(K, G)$, $\alpha' \in \check{Z}^1(K, \widehat{G})$. Because $\check{H}^2 \hookrightarrow {\rm{H}}^2$ in general, for each place $v$ of $K$ there exists a \v{C}ech cochain $\beta_v \in \check{C}^1(K_v, G)$ with $d\beta_v = \alpha_v$, where $\alpha_v$ denotes the image of $\alpha$ in ${\check{Z}}^2(K_v, G)$. 
In fact, we claim that for all but finitely many $v$, we may choose $\beta_v$ to be the image of an element of $\check{C}^1(\calO_v, \Gm)$, where $\calO_v \subset K_v$ denotes the ring of integers when $v$ is non-archimedean.
Indeed, in the number field case, let $X_S=\Spec(\calO_S)$ 
where $\calO_S\subset K$ denotes the ring of elements that are integral outside $S$,
and in the  function field case, if $K$ is the function field of a smooth projective curve $X$, let $X_S = X \smallsetminus S$.
Then for a sufficiently large  finite set $S$ of places of $K$, the cocycle $\alpha$ extends to an element in $\check{Z}^2(X_S, \Gm)$, 
hence (by abuse of notation) $\alpha_v \in \check{Z}^2(\calO_v, \Gm)$ for all $v \notin S$.
Because ${\rm{H}}^2(\calO_v, \Gm) = 0$ \cite[Ch.\,IV, Cor.\,2.13]{milneetale}, and again using the general fact that $\check{H}^2 \hookrightarrow {\rm{H}}^2$, we deduce that $\alpha = d\beta_v$ for some \v{C}ech cochain $\beta_v \in \check{C}^1(\calO_v, \Gm)$. This proves the claim.

We now obtain from this data an element of $\Q/\Z$ as follows. By \cite[Lem.\,5.11.1]{Rosengarten}, there exists $h \in \check{C}^2(K, \Gm)$ such that $\alpha \cup \alpha' = dh$. For each place $v$, one readily computes that $(\beta_v \cup \alpha'_v) - h \in \check{C}^2(K, \Gm)$ is a $2$-cocycle, hence -- via the natural map $\check{H}^2 \rightarrow {\rm{H}}^2$ -- we may take its invariant. Then we add these invariants up over all places of $v$ to get an element of $\Q/\Z$, which we shall denote $\langle \alpha, \eta', (\beta_v)_v\rangle_K \in \Q/\Z$, defined for any triple consisting of a \v{C}ech cocycle $\alpha \in \check{Z}^2(K, G)$, a cohomology class $\eta' \in {\rm{H}}^1(K, \widehat{G})$, and for all $v$ \v{C}ech cochains $\beta_v \in \check{C}^1(K_v, G)$ such that $d\beta_v = \alpha_v$ and such that $\beta_v$ extends to an element of $\check{C}^1(\calO_v, G)$ for all but finitely many $v$. When $\eta' \in \Sha^1(K, \widehat{G})$, this is none other than the $\Sha$ pairing (\ref{shapairing}), as follows from the definition of this pairing in \cite[\S5.11]{Rosengarten}.

It is crucial here to note that -- unlike with the $\Sha$ pairing as defined in \cite{Rosengarten} -- {\em this pairing depends on the choice of $\{\beta_v\}_v$} because we have not assumed that $\eta' \in \Sha^1$. Nevertheless, we claim that this pairing is well-defined -- that is, the sum defining it involves only finitely many nonzero terms and is independent of our choices of $h$ and $\alpha'$. The independence of $h$ follows from the fact that the sum of the invariants of a global Brauer class vanishes. To verify independence of the choice of $\alpha'$, let $\alpha' + d\gamma$ be another cocycle representing $\eta'$ with $\gamma \in \check{C}^0(K, \widehat{G})$. Then
\[
\alpha\cup(\alpha'+d\gamma) = \alpha\cup\alpha' + d(\alpha\cup\gamma) = d(h + (\alpha\cup\gamma)).
\]
Thus the new local Brauer class at $v$ obtained by the replacement $\alpha' \mapsto \alpha'+d\gamma$ differs from the old one by
\[
(\beta_v\cup d\gamma) - (\alpha\cup\gamma) = [(d\beta_v\cup\gamma) - d(\beta_v\cup \gamma)] - (\alpha\cup \gamma) = d(-\beta_v\cup\gamma),
\]
hence has the same invariant. This proves independence of the choice of cocycle $\alpha'$ representing $\eta'$.

It only remains to verify that all but finitely many of the local invariants involved in the pairing vanish. In order to show this, extend $\alpha'$ to a class in ${\check{Z}}^1(\calO_S, \widehat{G})$ and $h$ to a class in ${\check{C}}^2(\calO_S, \Gm)$ for some finite set $S$ of places of $K$ containing all of the archimedean places. For a nonarchimedean place $v$ of $K$, let $\calO_v \subset K_v$ denote the ring of integers. Enlarging $S$, we also have that $\beta_v$ extends to an element of $\check{C}^1(\calO_v, G)$ for all $v \notin S$. Thus, for $v \notin S$, the local Brauer class represented by $(\beta_v\cup \alpha')-h$ extends to an element of ${\rm{H}}^2(\calO_v, \Gm) = 0$ \cite[Ch.\,IV, Cor.\,2.13]{milneetale}. This completes the proof that $\langle \alpha, \eta', (\beta_v)_v\rangle_K$ is well-defined.

We will require the following lemma.

\begin{lemma}
\label{Brlemma}
Let $K$ be a global field. Then for every integer $n > 0$, the map $\Br(k)[n] \rightarrow \Br(k)/n\cdot\Br(k)$ is surjective. When $n$ is odd or $K$ has no real places, $\Br(k)/n\cdot\Br(k) = 0$.
\end{lemma}

\begin{proof}
The description of $\Br(k)$ provided by class field theory shows, for any nonarchimedean place $v_0$ of $K$, we have an isomorphism $$\Br(k) \simeq \left(\bigoplus_{v \mbox{ real}} (1/2)\Z/\Z\right) \oplus \left(\bigoplus_{v \neq v_0 \mbox{ nonarch}} \Q/\Z\right),$$ where the first sum is over the real places of $K$ and the second over the nonarchimedean places distinct from $v_0$. The lemma follows from this.
\end{proof}

With these preliminaries in hand, we are now prepared to carry out the second stage in the proof of Theorems \ref{charpinfinitemult} and \ref{charpinfinitequantitativemult}.

\begin{proposition}
\label{killingsha}
Let $\ell$ be a prime, $L/K$ an algebraic extension of a global function field such that either (1) $L/K$ is a locally infinite Galois $\ell$-primary extension, or (2) $\ell = p = {\rm{char}}(K)$ and $L$ is the maximal abelian $p$-primary extension of $K$. Let $M$ be a finite $\ell$-primary multiplicative $K$-group scheme which splits over a Galois extension of $K$ of degree $m$, and let $\xi \in \Sha^2(K, M)[\ell]$. Then there is some $K$-finite subextension $L/H/K$ such that $\xi$ vanishes over $H$. In case (2), we may choose $H$ so that ${\rm{log}}_p([H: K]) \leq 1+cm{\rm{log}}(m)^3$ for some universal effectively computable $c > 0$.
\end{proposition}

\begin{proof}[Proof]
As before, let $E := \widehat{M}$. By Lemma \ref{exhaust Sha}(i), there is a subextension $L/F/K$ of degree $\leq m!$ over $K$ and a finite subgroup $A \subset {\rm{H}}^1(F, E)$ such that, for all $F$-finite extensions $L/H/F$, $\Sha^1(H, E)$ is contained in the image of $A$ under the map
\[
{\rm{H}}^1(F, E) \rightarrow {\rm{H}}^1(H, E),
\]
and furthermore, every element of $A$ does indeed map to an element of $\Sha^1(H, E)$ for some finite intermediate $L/H/F$.

As in the definition of the pairing $\langle\cdot, \cdot, \cdot\rangle$, choose a cocycle $\alpha \in \check{Z}^2(F, M)$ representing $\xi$ and for each place $v$ of $F$ a cochain $\beta_v \in \check{C}^1(F_v, M)$ such that
\begin{equation}
\label{coceqn3}
d\beta_v = \alpha_v
\end{equation}
and such that $\beta_v$ extends to an element of $\check{C}^1(\calO_v, M)$ for all but finitely many $v$. Also choose for each of the finitely many $a \in A$ a cocycle $\alpha'_a \in \check{Z}^1(F, E)$ representing $a$ and a cochain $h_a \in \check{C}^2(F, \Gm)$ such that
\begin{equation}
\label{coceqn1}
\alpha\cup\alpha'_a = dh_a.
\end{equation}
Because $\alpha$ represents the $\ell$-torsion cohomology class $\xi$, there is $\tau \in \check{C}^1(K, M)$ such that
\begin{equation}
\label{coceqn2}
\ell \alpha = d\tau.
\end{equation}
For $a \in A$, let
\begin{equation}
\label{coceqn5}
\zeta_a := \tau\cup\alpha' - \ell h_a.
\end{equation}
Then
\begin{align*}
d\zeta_a & = d\tau\cup\alpha' -\tau\cup d\alpha' - \ell(dh_a) \\
& = \ell\alpha\cup\alpha' - \ell(\alpha\cup\alpha') = 0 \hspace{.2 in} \mbox{by (\ref{coceqn2}) and (\ref{coceqn1}),}
\end{align*}
so $\zeta_a$ represents a Brauer class. On the other hand, we are free to modify $h_a$ by a cocycle and preserve its defining relation (\ref{coceqn1}). This has the effect of modifying the cohomology class of $\zeta_a$ by the (negative of the) corresponding element of $\ell(\Br(k))$. Thanks to Lemma \ref{Brlemma}, therefore, we may assume that the Brauer class of $\zeta_a$ is trivial. Finally, define
\begin{equation}
\label{coceqn7}
\gamma_v := \ell \beta_v - \tau.
\end{equation}
Then we claim that
\begin{equation}
\label{coceqn4}
d\gamma_v = 0.
\end{equation}
Indeed, $d\gamma_v  = \ell(d\beta_v) - d\tau = 0$, the last equality due to (\ref{coceqn3}) and (\ref{coceqn2}).

For a place $v$ of a global field $H$, let ${\rm{inv}}_v\colon \Br(H_v) \rightarrow \Q/\Z$ denote the invariant map. Let $S$ be the (finite!) set of places of $F$ at which ${\rm{inv}}_v((\beta_v\cup\alpha'_a)-h_a) \neq 0$ for some $a \in A$. Then one may find an $F$-finite extension $L/F'/F$ such that every $a \in A$ becomes trivial at every place of $F'$ above $S$. In case (1), one simply passes to a Galois extension $F'/F$ such that this happens for some place above every $v \in S$, and then Galoisness ensures that it happens at every place above such $v$. In case (2), Lemma \ref{exhaust Sha}(ii) says that we can find such $F'$ with
\begin{equation}
\label{degbound2}
{\rm{log}}_p([F': K]) \leq cm{\rm{log}}(m)^2.
\end{equation}
Because the Brauer classes $(\beta_v\cup\alpha'_a) - h_a$ vanish for $v \notin S$ and all $a \in A$, it follows that, for any finite extension $F''/F'$ and any $a \in A$,
\begin{equation}
\label{coceqn9}
\langle\alpha, a, (\beta_v)_v\rangle_{F''} = \sum_{w\mid S} {\rm{inv}}_w((\beta_v\cup \alpha'_a) - h_a),
\end{equation}
where the sum is over the places of $F''$ above $S$.

Now choose $F'$-finite $L/F''/F'$ such that, for each place $v$ of $F'$ above $S$, and each place $w$ of $F''$ above $v$,
\begin{equation}
\label{coceqn10}
\ell \mid [F''_w: F'_v].
\end{equation}
In case (1), such $F''$ exists because $L/K$ is a locally infinite $\ell$-primary extension (where again we use Galoisness to ensure that the required divisibility happens at every place above a given place $v$ once it happens at a single place above $v$). In case (2), Theorem \ref{realizinglocexts} and (\ref{degbound2}) provide such an $F''$ with
\[
{\rm{log}}_p([F'': K]) \leq {\rm{log}}_p([F': K]) + (1+cm{\rm{log}}(m)^2)(1+{\rm{log}}[1+cm{\rm{log}}(m)^2]) \leq 1 + c'm{\rm{log}}(m)^3
\]
for some universal effectively computable $c' > 0$.

We claim that
\begin{equation}
\label{Shavscoarsepairing}
\langle \alpha, a, (\beta_v)_v\rangle_{F''} = 0
\end{equation}
for all $a \in A$. Assuming this, we claim that it will then follow that $\xi_{F''} = 0$, which will prove the proposition (upon taking $H := F''$). Indeed, all elements of $\Sha^1(F'', E)$ are of the form $a_{F''}$ for $a \in A$, and if $a_{F''}$ does in fact lie in $\Sha$, then the left side of (\ref{Shavscoarsepairing}) equals the $\Sha$ pairing between $\xi$ and $a_{F''}$ over $F''$. It then will follow that $\xi_{F''} \in \Sha^2(F'', M)$ annihilates $\Sha^1(F'', E)$ under the $\Sha$ pairing, and is therefore $0$ by \cite[Th.\,1.2.10]{Rosengarten}.

It remains to prove (\ref{Shavscoarsepairing}). We first claim that, for each place $v$ of $F'$ above $S$, we have
\begin{equation}
\label{coceqn8}
\ell\cdot{\rm{inv}}_v((\beta_v \cup \alpha')-h_a) = 0.
\end{equation}
Indeed, by our choice of $F'$, $\alpha'_v$ is trivial in cohomology for $v\mid S$, so
\begin{equation}
\label{coceqn6}
\alpha'_v = d\pi_v
\end{equation}
for some $\pi_v \in \check{C}^0(F'_v, E)$. We compute:
\begin{align*}
\ell[(\beta_v\cup\alpha')-h_a] & = [(\ell \beta_v)\cup\alpha']-\ell h_a \\
&= [(\gamma_v + \tau)\cup\alpha'] + [\zeta_a - (\tau\cup\alpha')] \hspace{.2 in} \mbox{by (\ref{coceqn7}) and (\ref{coceqn5})} \\
&= (\gamma_v\cup d\pi_v) + \zeta_a \hspace{.2 in} \mbox{by (\ref{coceqn6})}\\
&= d(-\gamma_v\cup \pi_v) + \zeta_a \hspace{.2 in} \mbox{by (\ref{coceqn4})}.
\end{align*}
This proves (\ref{coceqn8}), since $\zeta_a$ is trivial as a Brauer class.

Now we may prove (\ref{Shavscoarsepairing}). Thanks to (\ref{coceqn9}), it is enough to show that ${\rm{inv}}_w((\beta_v\cup\alpha')-h_a) = 0$ for each place $w \mid S$ of $F''$. But, if $v$ is the place of $F'$ below $w$, then
\[
{\rm{inv}}_w((\beta_v\cup\alpha')-h_a) = [F''_w: F'_v]\cdot{\rm{inv}}_v((\beta_v\cup\alpha')-h_a),
\]
and this vanishes by (\ref{coceqn8}) and (\ref{coceqn10}). This completes the proof of the proposition.
\end{proof}

We can now easily complete the proofs of Theorems \ref{charpinfinitemult} and \ref{charpinfinitequantitativemult}.

\begin{proof}[Proof of Theorems \ref{charpinfinitemult} and \ref{charpinfinitequantitativemult}]
By an easy induction, we may assume in both cases that $n = 1$. By Proposition \ref{canlandinsha}, there is a finite extension $H/K$ such that $\xi_H \in \Sha^2(H, M)$ with $H \subset L$ in the setting of Theorem \ref{charpinfinitemult}, and $H/K$ an abelian $p$-primary extension satisfying
\[
{\rm{log}}_p([H: K]) \leq 1 + cm{\rm{log}}(m)^3
\]
in the setting of Theorem \ref{charpinfinitequantitativemult}. By Proposition \ref{killingsha}, there is a finite $F/H$ such that $\xi_F = 0$, where $F \subset L$ in the setting of Theorem \ref{charpinfinitemult}, and $F/H$ is an abelian $p$-primary extension satisfying
\[
{\rm{log}}_p([F: H]) \leq 1 + cm{\rm{log}}(m)^3
\]
in the setting of Theorem \ref{charpinfinitequantitativemult}. This completes the proof of the theorems, modulo the minor technical issue in the case of Theorem \ref{charpinfinitequantitativemult} that, when $m = 1$, we obtain the bound ${\rm{log}}_p([F: K]) \leq 2$ rather than the desired bound of $1$. To remedy this, we note that when $m = 1$, $M$ is split, hence $\Sha^2(H, M) = 0$, so there is no need for the second step above in which we kill $\xi_H$ by passing to $F$.
\end{proof}

\begin{remark}
We note that there is another quantitative conclusion that may be obtained from the above proof of Theorem \ref{charpinfinitequantitativemult}, and which therefore also propagates to Theorem \ref{charpquantitative}: The solvable $p$-primary extension $F$ splitting $\xi \in \mathrm{H}^i(K, M)[p^n]$ can also be chosen to lie in a tower $F = F_{2n}/F_{2n-1}/\dots/F_1/F_0 = K$ such that each $F_{i+1}/F_i$ is abelian.
\end{remark}

We may now finally prove Theorems \ref{charpinfinite} and \ref{charpquantitative}.

\begin{proof}[Proof of Theorems \ref{charpinfinite} and \ref{charpquantitative}]
Thanks to Remark \ref{vanishesbeyonddeg2}, we only need to treat $i = 1, 2$. To streamline notation, let $\ell = p$ in the case of Theorem \ref{charpquantitative}. First suppose that $i = 2$, so $\xi \in {\rm{H}}^2(K, T)[\ell^n]$. Then $\xi$ is the image of an element $\eta \in {\rm{H}}^2(K, T[\ell^n])$. Because $T$ splits over a finite Galois extension of degree $m$, the same holds for $T[\ell^n]$. The desired result therefore follows from Theorems \ref{charpinfinitemult} and \ref{charpinfinitequantitativemult}.

Now we treat the case $i = 1$. Let $F/K$ be a finite Galois extension of degree $m$ splitting $T$. By \cite[Th.\,1.5.1]{ono}, we have an exact sequence
\[
0 \longrightarrow M \longrightarrow {\rm{R}}_{A/K}(\Gm) \longrightarrow T^n \times {\rm{R}}_{B/K}(\Gm) \longrightarrow 0
\]
for some $n > 0$, where ${\rm{R}}$ denotes Weil restriction, $M$ is a finite $K$-group scheme, and $A$ and $B$ are finite products of cyclic subextensions of $F$. Because ${\rm{R}}_{A/K}(\Gm)$ splits over $F$, so does the finite multiplicative group $M$. Consider $\xi$ as an element of ${\rm{H}}^1(T^n \times {\rm{R}}_{B/K}(\Gm))$ via one of the $T$ factors. The desired result now follows from Theorems \ref{charpinfinitemult} and \ref{charpinfinitequantitativemult}, together with the fact that ${\rm{H}}^1(K, {\rm{R}}_{A/K}(\Gm)) \simeq {\rm{H}}^1(A, \Gm) = 0$.
\end{proof}

\noindent \address
\vspace{.3 in}

\noindent \email

\end{document}